\documentclass{article}
\usepackage[utf8]{inputenc}
\usepackage{ mathrsfs }
\usepackage{comment}
\usepackage{tikz}
\usepackage{indentfirst}
\usepackage[hidelinks]{hyperref}
\usepackage{enumitem}
\usepackage{amsthm}
\usepackage{amsrefs}
\usepackage{thm-restate}
\usepackage{amsfonts}
\usepackage{ amssymb }
\usepackage{graphicx}
\usepackage{amsmath}
\usepackage[font=small,labelfont=bf]{caption}

\newcommand{\lp}{\left(}
\newcommand{\rp}{\right)}
\newcommand{\lb}{\left[}
\newcommand{\rb}{\right]}
\newcommand{\Per}{\mathcal P}
\newtheorem{theorem}{Theorem}[section]
\newtheorem{lemma}[theorem]{Lemma}
\newtheorem{definition}[theorem]{Definition}

\newtheorem{proposition}[theorem]{Proposition}
\theoremstyle{remark}
\newtheorem{exmp}{Example}[section]
\newtheorem*{remark}{Remark}
\newcommand{\Free}{\mathbb F}
\newcommand{\FreeRelPer}{\lp \Free,\Per\rp}
\newcommand{\mathdash}{\text{--}}

\title{On the Bowditch boundary of a free group with random cyclic peripheral groups}
\author{Aaron W. Messerla}
\date{}

\begin{document}

\maketitle
\begin{abstract}
The topology of the Bowditch boundary of a relatively hyperbolic group pair gives information about relative splittings of the group. It is therefore interesting to ask if there is generic behavior of this boundary. The purpose of this article is to show that in the case of a non-abelian free group with cyclic peripheral structure, there is not a generic case. This goal is accomplished by showing the minimal size of cut sets in the boundary increases as the lengths of the random words generating the peripheral structure go to infinity.
\end{abstract}

\section{Introduction}

For the entirety of this paper, $\Free$ will denote a finitely generated free group of rank $r\geq 2$, and $\Per$ will denote a peripheral structure on $\Free$ consisting of cyclic subgroups. The group $\Free$ acts on the Bowditch boundary, $\partial\FreeRelPer$, and each peripheral subgroup $P\in \Per$ stabilizes a unique point. Each parabolic point $p$ has a neighborhood $p\in U\subset \partial\FreeRelPer$ so that $U\setminus \{p\}$ has two components by work of Haulmark \cite{HaulmarkLocalCutPoints}. Cashen and Manning have shown that given a finitely generated non-cyclic free group, and a random multi-word determining a peripheral structure, the Bowditch boundary will be connected, contain no cut points, and be non-planar \cite{CashenManning}. Additionally, it was shown in \cite{CashenManning} that no two non-parabolic points in the boundary form a cut set. The results of \cite{CashenManning} are enough to show that $\Free$ does not split freely or cyclically relative to the random peripheral structure $\Per$. With this much known structure, it is tempting to ask if there is some generic case for $\partial\FreeRelPer$, where the words generating $\Per$ are chosen randomly. 

It is shown in \cite{CashenMacuraLinePatterns} that for a fixed cyclic peripheral structure $\Per$, $\partial\FreeRelPer$ has a finite cut set. However, the following theorem shows that there is no generic behavior of the Bowditch boundary in the case we are considering.

\begin{restatable}{theorem}{NoSizeK}\label{NoSizeK}
For any given positive integer $k$, and a random cyclic peripheral structure $\Per$, $\partial\FreeRelPer$ does not contain a cut set of size less than $k$.
\end{restatable}

Here, random is defined to be analogous to either the \emph{few relators} model of \cite{ArzhantsevaOlshanskiFewRelators}  or the \emph{density model} of Gromov \cite{GromovAsymptInv}, where instead of adding the words as relators, words are used to define a cyclic peripheral structure on the free group.

There are two main components to the proof of Theorem \ref{NoSizeK}. First, while Cashen and Manning rule out the possibility of splitting over a word relative to the peripheral structure, there is still the possibility that two parabolic points, or a parabolic point and a conical limit point, form a cut pair in the boundary. For an example of such behavior, one can consider $\Free=\langle a,b\rangle$  and $\Per=\{\langle\lb a,b\rb\rangle\}$. In this case, it is well known that $\partial\FreeRelPer=\mathbb S^1$, which has the property that any pair of points is a cut pair. These other cut pairs will not happen generically, shown by Lemma \ref{noweirdcutpairs}, which is proved with techniques similar to those in \cite{CashenManning}.  Secondly, it is shown that when every reduced word of length 2 occurs as a subword $k$ times in the peripheral structure, the boundary will not have cut sets of size less than $k$, which is shown in Lemma \ref{Kconnectedimpliesnosmall}. This lemma relies on several results of \cite{CashenMacuraLinePatterns}.

The paper is organized as follows. In Section \ref{WhiteheadGraphs}, generalized Whitehead Graphs are reviewed, as well as establishing some lemmas used in the proof of Theorem \ref{NoSizeK}. In Section \ref{Randomness}, the definition of exponentially generic sets will be reviewed, and the properties needed to apply Lemma \ref{Kconnectedimpliesnosmall} will be shown to be exponentially generic. The proof of the main theorem will follow shortly after.

\subsection*{Acknowledgments} 
\vspace{-.75pt} The author would like to thank his advisor, Daniel Groves, for suggesting this problem and many helpful discussions, as well as Chris Hruska for an insightful conversation about cut pairs in relative boundaries.

\section{Background and Generalized Whitehead Graphs}\label{WhiteheadGraphs}
\subsection{The Bowditch Boundary and the Decomposition Space}
In this paper, we are concerned with the Bowditch boundary of a free group, $\Free$ with cyclic peripheral structure $\Per$. Most existing literature on this topic describes this structure as a \emph{decomposition space} of a line pattern. It will be useful to switch between these two ideas throughout the paper, so we begin with a discussion of both topics.

The decomposition space was introduced in \cite{Otal} and studied further in \cites{CashenMacuraLinePatterns,CashenManning, CashenSplittingLinePatterns}. We take the following definitions from \cite{CashenMacuraLinePatterns}.
\begin{definition}[$g$-Lines]
For a cyclically reduced word $g$ which is not a proper power, a $g$\emph{-line} is the geodesic in a Cayley graph of $\Free$ containing the vertices $\{hg^{k}\}_{k\in \mathbb Z}$ for some $h\in\Free$.  
\end{definition}

We note that any left coset of $\langle g\rangle$ determines a $g$-line.

\begin{definition}[Line Patterns]
    Given a collection $\{g_1,...,g_n\}$ where each $g_i$ is cyclically reduced and not a proper power, the \emph{line pattern $\mathcal L$ generated by} $\{g_1,...,g_n\}$, is the collection of all $g_i$-lines in a Cayley graph of $\Free$.
\end{definition}
   
Let $\mathcal T$ be a Cayley graph for $\Free$, and let $\partial \mathcal T$ denote its boundary. Fix a line pattern $\mathcal L$. 

\begin{definition}
    The \emph{decomposition space} $D_\mathcal L$ associated to $\mathcal L$ is the quotient of $\partial\mathcal T$ where two points $x,y\in\partial\mathcal T$ are identified if there is a line $l\in\mathcal L$ with one endpoint $x$ and the other endpoint $y$.
\end{definition}

The following definition is given by Bowditch in \cite{BowditchRH}. 

\begin{definition}[Peripheral Structure]
A \emph{peripheral structure} on a group $G$ consists of a set $\Per$ of infinite subgroups of $G$ such that each $P\in \Per$ is equal to its normalizer in $G$ and each $G$-conjugate of a $P\in\Per$ lies in $\Per$.    
\end{definition}

The following proposition is well known, and will be used without reference when using a finite collection of words in $\Free$ to define a peripheral structure on $\Free$. For a word $w\in \Free$ with a fixed generating set $\mathcal B$, we let $\|w\|_\mathcal{B}$ denote the length of $w$. Let $w'$ denote the cyclic reduction of $w$, and $\sqrt{w}$ denote the shortest word so that $w=\sqrt{w}^i$ for some $i$. 

\begin{proposition}
    Let $\{g_1,...,g_n\}$ be a finite set of words in $\Free$. Then, the set of all conjugates of subgroups of the form $\langle\sqrt{g_i'}\rangle$ is a peripheral structure on $\Free$. Furthermore, $\Free$ is hyperbolic relative to this peripheral structure.
\end{proposition}
\begin{proof}
 Let $\Per$ be the set of all conjugates of subgroups of the form $\langle\sqrt{g_i'}\rangle$. Every cyclic subgroup of a free group is infinite, and every $G$ conjugate of $\langle\sqrt{g_i'}\rangle$ is in $\Per$ by construction. Each $\sqrt{g_i'}$ is cyclically reduced and not a proper power, so $\langle\sqrt{g_i'}\rangle$ is maximal cyclic, and conjugates of maximal cyclic subgroups are maximal cyclic. Since maximal cyclic subgroups of free groups are equal to their normalizer, $\Per$ is a peripheral structure on $\Free$. 

 As maximal cyclic subgroups are quasi-convex in $\Free$, and distinct maximal cyclic subgroups have trivial intersection, $\FreeRelPer$ is relatively hyperbolic by \cite{BowditchRH}*{Theorem 7.11}
\end{proof}

 We say that the peripheral structure formed by taking all conjugates of the $\langle \sqrt{g_i'}\rangle$ is the peripheral structure \emph{induced} by $\{g_1,...,g_n\}$, and that the words $\sqrt{g_i'}$ \emph{generate} $\Per$.

 When referencing a peripheral structure $\Per$, it may be useful to instead consider the line pattern $\mathcal L$ \emph{generated} by $\Per$. As above, consider $\Per$ as the conjugacy classes of finitely many cyclically reduced words $\{g_1,...,g_n\}$ where each $g_i$ is not a proper power.
 
 \begin{definition}
     The line pattern \emph{generated by} $\Per$ is the line pattern generated by the $\{g_1,...,g_n\}$ which generate $\Per$.
 \end{definition}
 
 By an \emph{endpoint} of a $g$-line, we mean a point in the boundary of the Cayley graph corresponding to either $hg^\infty$ or $hg^{-\infty}$

\begin{remark}
For a general relatively hyperbolic group pair, the Bowditch boundary was introduced in \cite{BowditchRH}. It is immediate from \cite{Tran}*{Theorem 1.1} that the Bowditch boundary of a free group with cyclic peripheral structure is equivalent to the decomposition space associated to the line pattern generated by the peripheral structure.  In this paper, $\partial\FreeRelPer$ will be discussed as a boundary, even when referencing the literature that discusses the decomposition space.
\end{remark}

\subsection{Generalized Whitehead Graphs}

As shown in \cite{CashenMacuraLinePatterns}, an efficient way of analyzing connectedness properties of the relative boundary is through the use of \emph{Whitehead graphs}. Let $\mathcal T$ be a Cayley graph of $\mathbb F$ with a fixed basis $\mathcal B$, and let $\overline{\mathcal T}$ denote its compactification by taking the union of $\mathcal T$ with its Gromov boundary  $\partial \mathcal T$, and the usual topology. Given a cyclic peripheral structure $\Per$, let $\mathcal L$ denote the line pattern generated by $\Per$. Let $\mathcal X$ be some closed connected subset of $\overline{\mathcal T}$.

\begin{definition}[Generalized Whitehead Graphs]
    The generalized Whitehead graph $\mathcal{WH}_{\mathcal{B}}\lp \mathcal X\rp\{\mathcal L\}$ is a graph  with vertices the connected components of $\overline{\mathcal T}\setminus \mathcal X$, and an edge between vertices $X$ and $Y$ if there is a line in $\mathcal L$ with endpoints in both $X$ and $Y$.
\end{definition}

It will be common to suppress the $\mathcal B$ and/or $\mathcal L$ when the basis and/or line pattern should be clear from context.

\begin{exmp}
    Consider the Whitehead graph $\mathcal{WH}\lp \mathcal \{p\}\rp\{\mathcal L\}$ for a vertex $p\in \mathcal T$. This graph consists of $2r$ vertices where $r$ is the rank of $\Free$ and some number of edges. The edges of $\mathcal{WH}\lp \mathcal \{p\}\rp\{\mathcal L\}$ correspond to every length 2 subword occurring in $\{g_1,...,g_n\}$, where each $g_i$ is considered as a cyclic word. The number of edges in $\mathcal{WH}\lp \mathcal \{p\}\rp\{\mathcal L\}$ is the sum of the lengths of the $g_i$. Since line patterns are invariant under translation, the choice of vertex is irrelevant.
\end{exmp}

\begin{figure}[h!]
     \centering\includegraphics[scale =1]{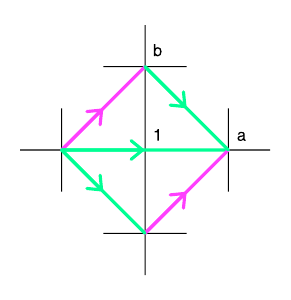}    
    \caption{
    Suppose $r=2$ and $\Free=\langle a,b\rangle.$ Let $\Per$ be generated by $\{ab, a^2b^{-1}\}$ and let $\mathcal L$ be the induced line pattern. In the illustration of $\mathcal{WH}\lp \{1\}\rp\{\mathcal L\}$ above, the purple edges indicate $ab$-lines and the green lines indicate $a^2b^{-1}$-lines. Note there are three $a^2b^{-1}$-lines going through $1$ corresponding to the left cosets $\langle a^2b^{-1}\rangle,\ a^{-1}\langle a^2b^{-1}\rangle,\text{ and }a^{-2}\langle a^2b^{-1}\rangle$. In $\mathcal T\setminus\{1\}$, these lines connect the components corresponding to $b$ with $a$, $a^{-1}$ with $a$, and $a^{-1}$ with $b^{-1}$ respectively. Note that to see this for the $\langle a^2b^{-1}\rangle$ coset, the line must be extended in both directions. Note additionally that these edges come from considering $a^2b^{-1}$ as a cyclic word and seeing the length 2 subwords $b^{-1}a$, $a^2$, and $ab^{-1}$ respectively. Similarly, the purple lines correspond to the length 2 subwords $ab$, and $ba$.
    }
    \label{WHEX}
\end{figure}

The following notation will be used to describe alterations to a Whitehead graph. For a Whitehead graph $\mathcal{WH}\lp \mathcal X\rp\{\mathcal L\}$, and a collection of lines $\tilde S\in \mathcal L$, we let $\mathcal{WH}\lp \mathcal X\rp\{\mathcal L\}\setminus \tilde S$ denote the graph obtained from $\mathcal{WH}\lp \mathcal X\rp$ by removing the interiors of edges corresponding to the lines in $\tilde S$. Note that if $\tilde S$  consists entirely of lines that do not go through a point $p$, then $\mathcal{WH}\lp \{p\}\rp\setminus\tilde S$ is the same as $\mathcal{WH}\lp \{p\}\rp$.

Given a graph $\Gamma$, and a vertex $v$ of $\Gamma$, we let $\lp \Gamma\rp\mathdash v$ be a graph where the vertex $v$ has been deleted, but the edges incident on $v$ are retained, with new, distinct, terminal vertices added. We call such an edge a \emph{loose edge}. The new terminal vertices will largely be ignored, and are removed after the process of splicing described below. 

\begin{remark}
    This is a departure from the notation in \cite{CashenMacuraLinePatterns}. While Cashen and Macura use the symbol ``$\setminus$" to indicate all of the described actions, in this paper ``$\setminus$" will be used for set difference, and removal of interiors of edges from Whitehead graphs, while ``$\mathdash$" will be used exclusively for the operation of removing a vertex from a graph, while retaining the loose edges.
\end{remark}

The notion of \emph{splicing} was introduced by Manning in \cite{ManningVirtuallyGeometric}. Given two disjoint closed subsets $\mathcal A,\mathcal B\subset \overline{\mathcal T}$ so that the distance from $\mathcal A$ to $\mathcal B$ is 1, we can construct $\mathcal{WH}\lp \mathcal A\cup \mathcal B\rp$ from $\mathcal{WH}\lp \mathcal A\rp$ and $\mathcal{WH}\lp \mathcal B\rp$ as follows. Suppose that $x$ and $y$ are vertices in $\mathcal A$ and $\mathcal B$ respectively so that $x$ and $y$ are joined by an edge $e$ in $\mathcal T$. Then $\mathcal T\setminus \mathcal A$ has one component corresponding to the edge $e$, and so does $\mathcal T\setminus \mathcal B$. In the Whitehead graphs, $e$ determines a vertex $v$ in $\mathcal{WH}\lp \mathcal A\rp$ and $w$ in $\mathcal{WH}\lp \mathcal B\rp$. An edge incident on $v$ in $\mathcal{WH}\lp \mathcal A\rp$ corresponds to a line in the pattern which goes through the vertices $x$ and $y$ in $\mathcal T$. The same is true of edges incident on $w$ in  $\mathcal{WH}\lp \mathcal B\rp$. In particular, in $\lp\mathcal{WH}\lp \mathcal A\rp\rp\mathdash v$ and $\lp \mathcal{WH}\lp \mathcal B\rp\rp\mathdash w$, we can then pair up a loose edge from $\lp \mathcal{WH}\lp \mathcal A\rp\rp\mathdash v$ with a loose edge from $\lp \mathcal{WH}\lp \mathcal B\rp\rp\mathdash w$, when the edges correspond to the same line. The resulting graph is $\mathcal{WH}\lp \mathcal A\cup \mathcal B\rp$.

\begin{figure}[h!]
     \centering\includegraphics[scale =1]{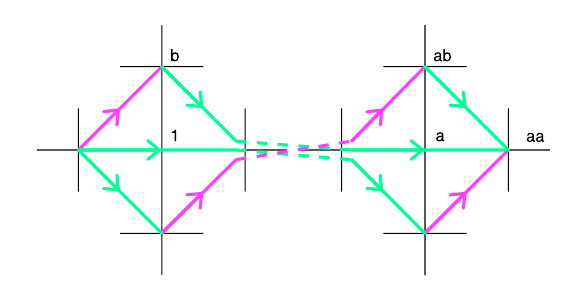}    
    \caption{Let $\mathcal L$ be as in Figure \ref{WHEX}. This is an illustration of $\mathcal{WH}\lp \lb 1,a\rb\rp\{\mathcal L\}$. The loose edges of $\mathcal{WH}\lp \{1\}\rp$ and $\mathcal{WH}\lp \{a\}\rp$ are those incident on the dotted segments, and the dotted segments indicate which lines are paired in the splicing process. }
    \label{SpliceExample}
\end{figure}

We can generalize the notion of splicing to cover the case where a set of lines in the family is ignored. In particular, let $\tilde S=\{l_1,...,l_k\}$ be a finite set of lines in $\mathcal L$. If $\mathcal A$ and $\mathcal B$ are disjoint subsets of $\mathcal T$ so that there are adjacent vertices $x\in\mathcal A$ and $y\in\mathcal B$ connected by an edge $e$, we have that a $l_i$ goes through $e$ if and only if it determines a line in both $\mathcal{WH}\lp \mathcal A\rp$ and $\mathcal{WH}\lp\mathcal B\rp$. Therefore, one could use splicing to recover $\mathcal{WH}\lp \mathcal A\cup\mathcal B\rp$, and then remove the edges corresponding to $\tilde S$ to get the graph $\mathcal{WH}\lp \mathcal A\cup\mathcal B\rp\setminus \tilde S$. Alternatively, in the same way as above, one can recover $\mathcal{WH}\lp \mathcal A\cup \mathcal B\rp\setminus \tilde S$ from splicing together $\mathcal{WH}\lp \mathcal A\rp\setminus \tilde S$ and $\mathcal{WH}\lp \mathcal B\rp\setminus \tilde S$.

The following idea is found in the proof of \cite{CashenMacuraLinePatterns}*{Lemma 4.20}, which is cited as Lemma \ref{CoreGood} below.

\begin{lemma}\label{AddConnected}
 Suppose $\mathcal A$ and $\mathcal B$ are bounded subsets of $\mathcal T$, such that there is an edge $e$ connecting $\mathcal A$ to $\mathcal B$. Let $\tilde S$ be a finite family of lines. If $\mathcal{WH}\lp \mathcal A\rp\setminus \tilde S$ is connected with no cut points, then $\mathcal{WH}\lp \mathcal A\cup \mathcal B\rp\setminus \tilde S$ contains the same number of connected components as $\mathcal{WH}\lp \mathcal B\rp\setminus \tilde S$.
\end{lemma}
\begin{proof}
    Suppose that $v\in\mathcal B$ is the vertex adjacent to $e$, and $w\in\mathcal A$ is adjacent to $e$. Note that $v$ determines a vertex of $Wh\lp \mathcal A\rp$ and $w$ a vertex in $\mathcal{WH}\lp \mathcal B\rp$.  Since $\mathcal{WH}\lp \mathcal A\rp\setminus\tilde S$ is connected with no cut vertices, $\lp \mathcal{WH}\lp \mathcal A\rp\setminus\tilde S\rp\mathdash v$ is connected. Therefore, the graph $\mathcal{WH}\lp \mathcal A\cup \mathcal B\rp\setminus \tilde S$ is the splicing of $\lp\mathcal{WH}\lp \mathcal A\rp\setminus\tilde S\rp\mathdash v$ with $\lp \mathcal{WH}\lp \mathcal B\rp\setminus\tilde S\rp\mathdash w$. By collapsing the portion of $\mathcal{WH}\lp \mathcal A\cup \mathcal B\rp\setminus \tilde S$ corresponding to $\lp\mathcal{WH}\lp \mathcal A\rp\setminus \tilde S\rp\mathdash v$ to a point, we do not change the number of connected components, but this results in the graph $\mathcal{WH}\lp \mathcal B\rp\setminus \tilde S$.
\end{proof} 

\begin{exmp}
    If $\mathcal{WH}\lp \{p\}\rp$ is connected with no cut vertices and $q$ is adjacent to $p$, then $\mathcal{WH}\lp \left[p,q\right]\rp$ is connected with no cut vertices by Lemma \ref{AddConnected}. Furthermore, as line patterns are invariant under translation, induction and the lemma show that $\mathcal{WH}\lp \mathcal A\rp$ is connected with no cut vertices for any bounded subtree $\mathcal A\subset \mathcal T$.
\end{exmp}

\subsection{Connectivity of the Boundary from Whitehead Graphs}

The following results, culminating in Lemma \ref{Kconnectedimpliesnosmall}, will be used to show that if $\Per$ is generated by a random collection of elements of $\Free$, then the size of the smallest cut sets in $\partial\FreeRelPer$ will increase with the length of the words generating $\Per$. Numerous results and ideas from Section 4 of \cite{CashenMacuraLinePatterns} will be used throughout, and the reader is referred to that paper for details on Lemmas \ref{HullGood}, \ref{CoreGood}, and \ref{cases}. Lemma \ref{noweirdcutpairs} is proved using similar ideas to that of \cite{CashenManning}*{Proposition 5.4}, but a stronger notion than that of full words is needed. This will be discussed in this subsection, starting with Definition \ref{FullRef}.

Suppose that $\partial\FreeRelPer$ is connected. A \emph{cut set} $S\subseteq \partial\FreeRelPer$ is a finite collection of points such that $\partial\FreeRelPer\setminus S$ is disconnected. The cut set $S$ is \emph{minimal} if there is no proper subset $S'\subsetneq S$ such that $S'$ is a cut set. A \emph{cut point} is a cut set of size one, and a \emph{cut pair} is a minimal cut set of size two. 

Fix a Cayley tree $\mathcal T$ for $\Free$. Let $q:\partial \mathcal T\rightarrow \partial\FreeRelPer$ be the quotient map. A point $x\in \partial\FreeRelPer$ is said to be \emph{good} if $q^{-1}\lp x\rp$ is two points (corresponding to the endpoints of a line in the pattern generated by $\Per$), and \emph{bad} otherwise. Given a set points $S\subseteq \partial\FreeRelPer$ we can determine connectivity of $\partial\FreeRelPer\setminus S$ by the connectivity of a Whitehead graph using the following lemma.

\begin{lemma}{\cite{CashenMacuraLinePatterns}*{Lemma 4.9}}\label{HullGood}
   Suppose $S\subseteq \partial\FreeRelPer$ is a non-empty collection of points which is not a single bad point, and let $\mathcal H$ be the convex hull of $q^{-1}\lp S\rp$. There is a bijection of components of $\mathcal{WH}\lp \mathcal H\rp$ and components of $\partial\FreeRelPer\setminus S$. 
\end{lemma}

Fullness was used in \cite{CashenManning} to show that $\Free$ doesn't split relative to $\Per$. To rule out the possibility of the other kinds of cut pairs in $\partial\FreeRelPer$, we need a slightly stronger concept.

\begin{definition}\label{FullRef}
    For a positive integer $k$, we say that a cyclically reduced word $w\in \Free$ is $k$\emph{-full} if $w$, considered as a cyclic word, contains as a subword every $\sigma\in\Free$ where $\sigma$ has length $k$.
\end{definition}

We note that if $w$ is $k$-full, then $w$ is $l$-full for any $l<k$. Additionally, the definition of \emph{full} in \cite{CashenManning} is slightly weaker than the definition of $3$-fullness in this paper (as either the word of length 3 or its inverse is needed to appear for the definition in \cite{CashenManning}.) An analogous definition as in \cite{CashenManning} would have worked in this paper, but the stronger definition makes the proof of the Lemma \ref{noweirdcutpairs} slightly easier.

\begin{lemma}[cf. \cite{CashenManning}*{Proposition 5.4}]\label{noweirdcutpairs}
    Suppose $\Per$ is a peripheral structure on $\Free$ containing a 4-full word. Then $\partial\FreeRelPer$ is connected, and does not contain any cut points or cut pairs.
\end{lemma}
\begin{proof}
    Since $\Per$ contains a 4-full word, \cite{CashenManning}*{Lemma 5.3} implies that $\Free$ does not split freely relative to $\Per$. Connectedness of $\partial\FreeRelPer$ follows from the fact that a disconnected relative boundary implies a free splitting relative to the peripheral structure \cite{BowditchRH}*{Proposition 10.1}. 
    
    Let $S\subseteq \partial\FreeRelPer$ be a cut point or cut pair. Let $\mathcal L$ be the pattern generated by $\Per$. There are 3 cases that need to be considered. When $q^{-1}\lp S\rp$ is two points, the result follows from \cite{CashenManning}*{Proposition 5.4}, since 4-full words are 3-full (which are full in the sense of \cite{CashenManning}). There are two other cases to consider. First, if $q^{-1}\lp S\rp$ contains three points, and the other if $q^{-1}\lp S\rp$ contains 4 points.

    When $q^{-1}\lp S\rp$ contains 3 points, the convex hull is a tripod, which we denote $\mathcal Y$. Let $p$ be the point of intersection of the three rays forming the tripod. Let $X$ and $Y$ be two vertices in $\mathcal{WH}\lp \mathcal Y\rp$. Let $x\subseteq X$ be the vertex closest to $\mathcal Y$ and define $y\in Y$ in the same way. Take $\gamma:\lb 0,D\rb$ to be the path $\mathcal T$ joining $x$ and $y$, parameterized to have unit speed. Let $s_i$ denote the label of the edge connecting $\gamma\lp i-1\rp$ to $\gamma\lp i\rp$. We note that $\gamma|_{\lb 1,D-1\rb}$ is a path through $\mathcal Y$.
    
    If $X$ and $Y$ are in the same component of $\mathcal Y\setminus \{p\}$, the proof follows very similarly as in \cite{CashenManning}. If $D<4$, since $\Per$ contains a 4-full word, there is a subword $s_1s_2$, or $s_1s_2s_3$, which contributes an edge connecting $X$ to $Y$ in $\mathcal{WH}\lp \mathcal Y\rp$. Suppose now that $D\geq 4$, as $\Per$ contains a 4-full word, there are subwords $s_1s_2c_2$, $c_{i-1}^{-1}s_ic_{i}$ for $2\leq i\leq D-2$, and $c_{D-2}^{-1}s_{D-1}S_D$, where $c_i$ is picked so that $c_{i-1}\neq s_{i}$ and $c_{i}\neq s_i^{-1}$, so that all such words are reduced. These subwords contribute edges creating a path in $\mathcal{WH}\lp \mathcal Y\rp$ between $X$ and $Y$, so that $X$ and $Y$ are in the same component of $\mathcal{WH}\lp \mathcal Y\rp$. 

    Now, suppose $X$ and $Y$ are contained in different components of $\mathcal T\setminus \{p\}.$ Then $p$ is in the image of $\gamma$, so suppose $p=\gamma\lp i\rp$. Using the argument above, there is a path in $\mathcal{WH}\lp \mathcal Y\rp$ given by $s_1s_2c_2$,...,$c_{i-2}s_{i-1}c_{i-1}$. We replace the next two edges by an edge given by the subword $c_{i-1} s_is_{i+1} c_{i+1}$, which exists due to 4-fullness. This ensures that $c_i$, which may have been picked to land on a different leg of $\mathcal Y$, and therefore not determined a vertex of $\mathcal{WH}\lp \mathcal Y\rp$, is omitted. We then continue with the path $c_{i+1}^{-1}s_{i+2}c_{i+2}$, $c_{i+2}^{-1}s_{i+3}$,..., $c_{D-2}s_{D-1}s_{D}$.

    The case where $q^{-1}\lp S\rp$ has 4 points is similar. When $q^{-1}\lp S\rp$ has 4 points, the convex hull of $q^{-1}\lp S\rp$ is either in the shape of the letter X or the letter H. In the first case, the point of intersection $p$ can be circumvented in the same way as above. In the second case, both points with a neighborhood isometric to a tripod can be worked around as above.
\end{proof}

\begin{remark}
    In the case where $q^{-1}\lp S\rp$ has 3 points, it is always possible to pick a $c_i$ adjacent to $p$ so that there is an edge $c_{i-1}^{-1} s_i c_i$ which does not end on a vertex in $\mathcal Y$. Alternatively, if $S=\{x,y\}$ and $q^{-1}\lp x\rp=\{x^+, x^-\}$ and $q^{-1}\lp y\rp =\{ y^+, y^-\}$, so $q^{-1}\lp S\rp$ has 4 points, and the convex hulls of $\{x^+,x^-\}$ and $\{y^+,y^-\}$ intersect at a unique point $p$, then the convex hull of $q^{-1}\lp S\rp$ is in the shape of the letter X. If additionally the rank of $\Free$ is $2$, such a $c_i$ cannot be picked, and 4-fullness must be used. Figure \ref{whX} is included to illustrate the method of proof.
\end{remark}

\begin{figure}
    \centering\includegraphics{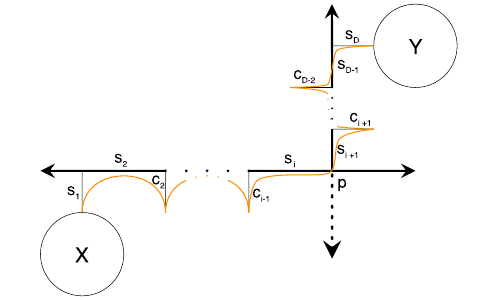}
    \caption{The set $\mathcal Y$ is illustrated by the bold black lines, where the dotted line is possibly included. The thin black lines specify edges in the Cayley graph which determine components of $\mathcal T\setminus \mathcal Y$. The orange arcs specify the edges in the Whitehead graph $\mathcal{WH}\lp \mathcal Y\rp$ connecting the components $X$ and $Y$.}
    \label{whX}
\end{figure}

In the case where $\partial\FreeRelPer$ is connected with no cut pairs, any finite minimal (with respect to inclusion) cut set is a collection of good points \cite{CashenMacuraLinePatterns}*{Lemma 4.19}. Suppose that $S$ is a finite minimal cut set of size $m\geq 3$, and $S$ is the image under $q$ of the endpoints of a collection of lines $\{l_1,...,l_m\}$ in the pattern generated by $\Per$. Let $\mathcal H$ be the convex hull of $q^{-1}\lp S\rp$. As in \cite{CashenMacuraLinePatterns}, we take the \emph{core} of $\mathcal H$ to be the smallest subtree $\mathcal C\subseteq \mathcal H$ such that $\mathcal H\setminus \mathcal C$ is a collection of disjoint rays. By Lemma \ref{HullGood} there is a bijection between components of $\mathcal{WH}\lp\mathcal H\rp$ and the components of $\partial\FreeRelPer\setminus S$. This is sharpened by the following lemma. Let $\tilde{S}$ denote the lines in the pattern generated by $\Per$ with endpoints in $q^{-1}\lp S\rp$.

\begin{lemma}{\cite{CashenMacuraLinePatterns}*{Proposition 4.20}}\label{CoreGood}
Let $S$ be a finite collection of good points, none of which is a cut point. Let $\mathcal H$ denote the convex hull of $q^{-1}\lp S\rp$, and $\mathcal C$ be the core of $\mathcal H$. Then there is a bijection between components of $\partial\FreeRelPer\setminus S$ and components of $$\mathcal{WH}\lp \mathcal C\rp\setminus \tilde{S}.$$
\end{lemma}

Cashen and Macura define the \emph{pruned core} as follows. Suppose that $\partial\FreeRelPer$ is connected so that for every $v\in \mathcal T$ we have that $\mathcal{WH}\lp \{v\}\rp$ is connected without cut points. Suppose $v$ is a degree 1 vertex in the core $\mathcal C$ of $S$. Call this vertex a \emph{leaf} and let the \emph{stem} be the edge connecting $v$ to the rest of $\mathcal C$. Denote the stem of $v$ by $st\lp v\rp$. As $st\lp v\rp$ also determines a vertex in $\mathcal{WH}\lp \{v\}\rp$, we will also use $st\lp v\rp$ to denote this vertex. If every line of $\mathcal S$ which appears in $\mathcal{WH}\lp \{v\}\rp$ goes through the stem, $st\lp v\rp$, then $\mathcal{WH}\lp \{v\}\rp\mathdash st\lp v\rp$ is connected, and furthermore $\lp\mathcal{WH}\lp \{v\}\rp\setminus \tilde S\rp\mathdash st\lp v\rp$ is connected. Therefore, $\mathcal{WH}\lp \mathcal C\setminus \{p\}\rp\setminus \tilde S$ has the same number of components as $\mathcal{WH}\lp \mathcal C\rp\setminus \tilde S$, so we can replace $\mathcal C$ with $\mathcal C\setminus\lp\{v\}\cup \textrm{Int}\lp st\lp v\rp\rp\rp$ and repeat the argument, pruning all leaves where each line of $\tilde S$ going through the leaf goes through the stem. At some point if there are two adjacent leaves, with all lines of $\tilde S$ running through both, then $S$ is a collection of all lines through an edge $e$, and Cashen and Macura retain this edge in the pruned core. What is left is neatly characterized by the following lemma.

\begin{lemma}{\cite{CashenMacuraLinePatterns}*{Lemma 4.22}}\label{cases}
Suppose there are no cut points or cut pairs in $\partial\FreeRelPer$. If $S$ is a cut set with pruned core $p\mathcal C$, then one of the following occurs:\begin{itemize}
    \item $S$ is the collection of endpoints of lines through an edge $e$, and $p\mathcal C$ is that edge
    \item $p\mathcal C$ is a single vertex
    \item $p\mathcal C$ has leaves and through every leaf there is a line in $\tilde S$ that does not go through the stem
\end{itemize}
    
\end{lemma}

The following lemma is a consequence of the facts summarized above. The proof of Lemma \ref{Kconnectedimpliesnosmall} when $p\mathcal C$ has leaves is summarized by the sketch in Figure \ref{whPandwhp}.

\begin{lemma}\label{Kconnectedimpliesnosmall}
Suppose that $\Per$ is a peripheral structure on $\Free$ containing a $4$-full word. Let $\mathcal L$ be the line pattern generated by $\Per$, and suppose $\mathcal{WH}\lp \{p\}\rp\{\mathcal L\}$ has the property that there are at least $k$ edges between each pair of vertices. Suppose $S$ is a minimal cut set of $\partial\FreeRelPer$. Then $S$ has at least $k$ elements.
\end{lemma}
\begin{proof}
  Since $\Per$ contains a $4-$full word, $\partial\FreeRelPer$ does not contain cut points or cut pairs by Lemma \ref{noweirdcutpairs}. Let $p\mathcal C$ denote the pruned core of $S$. By Lemma \ref{CoreGood} and the following discussion, we know that $\mathcal{WH}\lp p\mathcal C\rp\setminus \tilde S$ is disconnected. We can then look at each case from Lemma \ref{cases}.
  
  Suppose $p\mathcal C$ is an edge, and $\tilde S$ is all lines running through that edge. Suppose the edge is $e$ and is labeled by the generator $a$. Then, the number of lines running through that edge is equal to the number of lines in $\mathcal{WH}\lp \{p\}\rp$ incident on the vertex corresponding to $a$. By assumption, there are more than $k$ edges incident to that vertex, so $S$ has more than $k$ elements.
  
  Suppose now that $p\mathcal C$ is a point $p$. Then, since $\mathcal{WH}\lp \{p\}\rp$ has at least $k$ edges connecting any pair of vertices, it remains connected after removing the interiors of any set of fewer than $k$ edges. Since $\mathcal{WH}\lp \{p\}\rp\setminus \tilde{S}$ is disconnected, there must be more than $k$ elements of $S$.
  
  Finally, suppose $p\mathcal C$ has leaves and through each leaf there is at least one line of $\tilde S$ not through the stem. Consider a leaf $v$, and the graph $\mathcal{WH}\lp \{v\}\rp\setminus \tilde S$. We know that $\mathcal{WH}\lp p\mathcal C\rp\setminus \tilde S$ is formed by splicing $\mathcal{WH}\lp \{v\}\rp\setminus \tilde S$ to $\mathcal{WH}\lp p\mathcal C\setminus \{v\}\rp\setminus \tilde S$. Let $x$ denote the vertex in $\mathcal{WH}\{v\}$ corresponding to the stem of $v$. By Lemma \ref{AddConnected}, if $\mathcal{WH}\{v\}\setminus \tilde S$ is connected with no cut points, then  $\mathcal{WH}\lp p\mathcal C\rp\setminus \tilde S$ has the same number of components as $\mathcal{WH}\lp p\mathcal C\setminus \{v\}\rp\setminus \tilde S$.
  
  Take $\tilde S'\subsetneq \tilde S$ to be the proper subset of $\tilde S$ which does not contain the lines through $v$ that don't pass through $st\lp v\rp$. Let $S'\in \partial\FreeRelPer$ be the image of the set of endpoints of $\tilde S'$ under $q$. The pruned core of $S$ contains the pruned core of $S'$, and $\mathcal{WH}\lp p\mathcal C\rp\setminus \tilde S'$ is disconnected. Therefore, the collection of points $S'\subset\partial \FreeRelPer$ is a cut set contained in $S$ which is strictly smaller than $S$. This contradicts the minimality of $S$. Therefore, $\mathcal{WH}\lp\{v\}\rp\setminus \tilde S$ must either be disconnected or contain a cut point. If $S$ contains fewer than $k$ points, then there is at least one edge between any pair of points in $\mathcal{WH}\lp\{v\}\rp\setminus \tilde S$, so this is impossible. Therefore, $S$ must have size at least $k$.
  \end{proof}

\begin{figure}[h!]
     \centering\includegraphics[width=.75\textwidth]{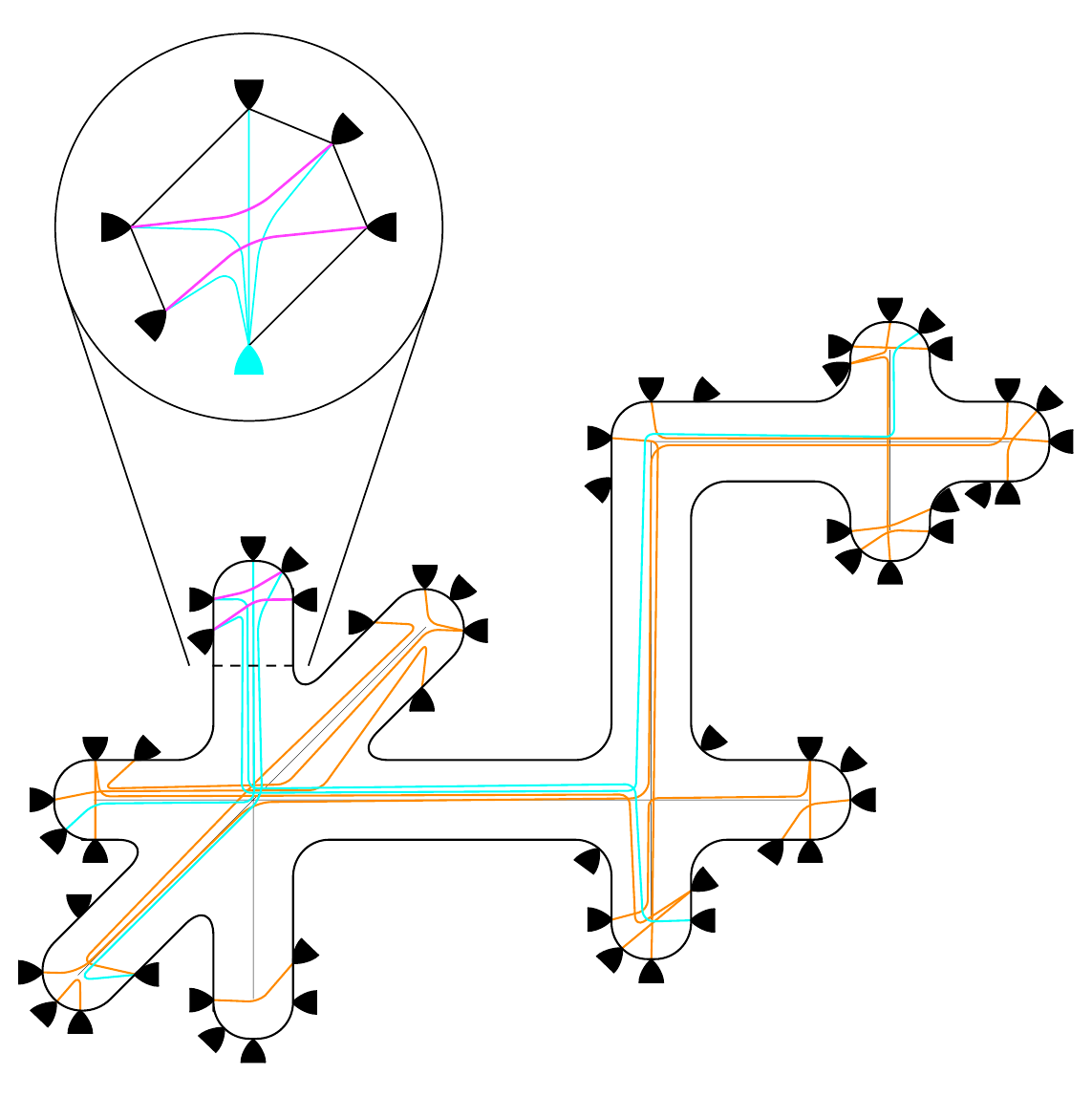}    
    \caption{The triangles represent components of $\mathcal T\setminus X$ where $X=p\mathcal C$ in the main image, and $X=v$ in the top left. These are the vertices in the appropriate Whitehead graphs. The blue, purple, and orange lines represent the elements of $\tilde S$. If the blue and purple lines with the blue vertex don't disconnect $\mathcal{WH}\lp\{v\}\rp$ (pictured top left), then the purple lines are not a necessary part of the cut set, so $S$ is not minimal.}
    \label{whPandwhp}
\end{figure}

\section{Randomness}\label{Randomness}
Following \cite{CashenManning}, we will use the fact that certain properties are \emph{exponentially generic} to show that those properties occur in any setting where the probability is chosen uniformly randomly. 

\begin{definition}[Exponentially Generic]
    Let $A\subset B\subset \Free$ be subsets and let $\mathcal B$ be a fixed basis of $\Free$. The subset $A$ is \emph{generic} in $B$ if $$\lim_{n\rightarrow\infty}\frac{\#\{w\in A|\|w\|_{\mathcal B}\leq n\}}{\#\{w\in A|\|w\|_{\mathcal B}\leq n\}}=1.$$
    $A$ is \emph{exponentially generic} in $B$ if there exists $c>0$ and $b\in \mathbb R$ such that for sufficiently large $n$, $$1-\frac{\#\{w\in A|\|w\|_{\mathcal B}\leq n\}}{\#\{w\in A|\|w\|_{\mathcal B}\leq n\}}\leq e^{b-cn}.$$ 
\end{definition}

The complement $B\setminus A$ is \emph{exponentially rare} in $B$ if $A$ is exponentially generic in $B$. We will use without reference the fact that if $A$ is exponentially rare in $B$, then there is a $c>0$ and $b\in \mathbb R$ such that for sufficiently large $n$, we have $$\frac{\#\{w\in A|\|w\|_\mathcal{B}\}}{\#\{w\in A|\|w\|_\mathcal{B}\}}\leq e^{b-cn}.$$

We remark that the intersection of exponentially generic sets is exponentially generic, and that if $A\subset B\subset C$ with $A$ exponentially generic in $C$, then $B$ is exponentially generic in $C$.

We use $\mathcal C$ to denote the set of cyclically reduced words in $\Free$.

\begin{lemma}[cf \cite{CashenManning}*{Lemma 2.4}]\label{ShellOnly}
     Suppose $A\subset \Free$ has the property that $$\frac{\#\{x\in A| \|x\|_{\mathcal B}=n\}}{\#\{x\in \mathcal C|\|x\|_\mathcal{B}=n\}}\leq e^{b-cn}$$ for some $c>0$ and $b\in \mathbb R$ and all sufficiently large $n$. Then $A$ is exponentially rare in $\mathcal C$.
\end{lemma}
\begin{proof}
    We let $A_n$ denote $\{ x\in\mathcal A|\|x\|_\mathcal{B}\leq n\}$ and $A_n^\circ$ denote $\{ x\in\mathcal A|\|x\|_\mathcal{B}= n\}$. We define $\mathcal C_n$ and $\mathcal C_n^\circ$ similarly. By assumption, we have that for sufficiently large $n$, \begin{equation}\label{ShellOnlyAssumption}
    \frac{\# A_n^\circ}{\#\mathcal{C}_n^\circ}\leq e^{b-cn}
    \end{equation}
    and we need to establish that $\frac{\#A_n}{\#\mathcal C_n}\leq e^{b'-c'n}$ for all sufficiently large $n$ and some $b'\in\mathbb R$ and $c'>0$.

    First, we have that $\#\mathcal{C}_n^\circ = \lp 2r-1\rp^n +r + \lp -1\rp^n\lp r-1\rp$ (see, for example \cite{MannHGG}*{Proposition 17.2}). Since $r\geq2$, for $m>k$, we have that $\# \mathcal{C}_m^\circ>\lp 2r-1\rp^m$ and $\#\mathcal{C}_k^\circ\leq 2\lp 2r-1\rp^k$, so we have that
     \begin{align*} 
      \frac{\#\mathcal{C}_m^\circ}{\#\mathcal{C}_k^\circ}&\geq\frac{\lp 2r-1\rp^m}{2\lp 2r-1\rp^k}\\
        &=\frac{1}{2}\lp 2r-1\rp^{m-k}
     \end{align*} 
    so that $\#\mathcal{C}_m^\circ\geq \frac{1}{2}\lp 2r-1\rp^{m-k} \#\mathcal{C}_k^\circ $

    Let $k_0$ be large enough that for $n>k_0$ we have that (\ref{ShellOnlyAssumption}) holds, and take $N<k_0$ to be the index such that $\#A^\circ_N$ is maximal in the finite set $\{\#A^\circ_i\}_{i=0}^{k_0}.$ Then, the following holds for $n>k_0$:
    \begin{align*}
        \frac{\# A_n}{\# \mathcal{C}_n}&=\frac{\sum_{i=1}^n \#A_n^\circ}{\#\mathcal{C}_n}\\
        &\leq \frac{\sum_{i=k_0}^n \#A_i^\circ}{\#\mathcal{C}_n^\circ}+\frac{\sum_{i=0}^{k_0} \#A_i^\circ}{\#\mathcal{C}_n}\\
        &\leq \sum_{i=k_0}^{n}\frac{\#A_i^\circ}{\frac{1}{2}\lp 2r-1\rp^{n-i}\#\mathcal{C}_i^\circ} +\frac{k_0 \#A_N}{\#\mathcal{C}_n}\\
        &\leq\sum_{i=k_0}^n\frac{2e^{b-ci}}{\lp 2r-1\rp^{n-i}}+\frac{k_0\#A_N}{\lp 2r-1\rp^n}\\
        &=\sum_{i=k_0}^n e^{\ln\lp 2\rp+b-ci-\lp n-i\rp \ln{\lp 2r-1\rp}}+e^{\ln{\lp k_0\#A_N\rp}-\ln{\lp 2r-1\rp}n}\\
        &\leq\sum_{i=k_0}^n e^{b'+\lp\ln{\lp 2r-1\rp}-c'\rp i-n\ln\lp 2r-1\rp}+e^{\ln{\lp k_0\#A_N\rp}-c'n}\\
        &\leq \lp n-k_0\rp e^{b+\lp\ln\lp 2r-1\rp-c'\rp n-n\ln\lp 2r-1\rp}+e^{\ln{\lp k_0\#A_N\rp}-c'n}\\
        &= e^{b'-c'n+\ln\lp n-k_0\rp}+e^{\ln{\lp k_0\#A_N\rp}-c'n}\\
        &\leq e^{b'-c'n+\ln\lp n\rp}+e^{\ln\lp k_0\#A_N\rp-c'n}\\
        &\leq e^{b'-c''n}+e^{\ln\lp k_0\#A_N\rp-c''n}\\
        &\leq e^{b''-c''n}
    \end{align*}

where $b'=\ln\lp 2\rp+b$, $c'\leq \min\{ c, \ln\lp 2r-1\rp\}$, $c''<c'$ and $b''\geq \ln\lp 2\rp+\max\{b', \ln\lp k_0\#A_N\rp\}$, when $n$ is sufficiently large.
\end{proof}

The converse of Lemma \ref{ShellOnly} will also be useful in the proof of Lemma \ref{rarereduction}.

\begin{lemma}\label{CM24converse}
    Using the notation as in Lemma \ref{ShellOnly}, if $A$ is exponentially rare in $\mathcal C$, then there are $b'\in \mathbb R$, $c'>0$ and $M'$ such that for all $n>M'$, $$\frac{\#A_n^\circ}{\#\mathcal C_n^\circ}\leq e^{b'-c'n}.$$
\end{lemma}
\begin{proof}
    In order to obtain a contradiction, suppose that for any $b'\in \mathbb R$, $c'>0$, and $M'$, there is an $n>M'$ so that $\frac{\#A_n^\circ}{\#\mathcal C_n^\circ}>e^{b'-c'n}$. Let $b\in\mathbb R$, $c>0$, and $M$ be such that for all $n>M$, we have $\frac{\#A_n}{\#\mathcal C_n}\leq e^{b-cn}$. Pick $b'>b$, $c'$ such that $0<c'<c$, and $M'$ so that for all $n>M'$, we have $-\ln\lp n+1\rp-c'n\geq cn$. Take $M''$ so that $M''>\max\{M,M'\}$. Then pick an $n>M''$ so that $\frac{\#A_n^\circ}{\#\mathcal C_n^\circ}>e^{b'-c'n}$. Then we have that \begin{align*}
        \frac{\#A_n}{\#\mathcal C_n}&\geq \frac{\#A_n^\circ}{\sum_{i=0}^n\#\mathcal C_n^\circ}\\
        &>\frac{\#A_n^{\circ}}{\lp n+1\rp\#C_n^\circ}\\
        &>e^{b'-c'n-\ln\lp n+1\rp}\\
        &>e^{b-cn}
    \end{align*}
    which is the required contradiction.
\end{proof}

Recall that for a reduced word $\Gamma$, we take $\Gamma'$ to be the cyclic reduction, and $\sqrt{\Gamma}$ to be the generator of the maximal cyclic subgroup containing $\Gamma$.

The following lemma appears without proof in \cite{KSSGenericProperties}. For completeness we include a proof here.

\begin{lemma}[\cite{KSSGenericProperties}*{Proposition 6.2(2)}]\label{rarereduction}
    Suppose that $B$ is exponentially generic in $\mathcal C$. Let $A$ be the set of words $w\in \Free$ such that $w'\in B$. Then $A$ is exponentially generic in $\Free$.
\end{lemma}
\begin{proof}
   It is sufficient to show that if $B$ is exponentially rare in $\mathcal C$, then $A$ is exponentially rare in $\Free$. Suppose that $b\in \mathbb R$ and $c>0$ are such that $\frac{\#B_n}{\#\mathcal C_n}\leq e^{b-cn}$ for sufficiently large $n$.
   
   By \cite{CashenManning}*{Lemma 2.4}, it is sufficient to show that there are $b'\in\mathbb R$ and $c'>0$ such that $\frac{\#A_n^\circ}{\#\Free_n^\circ}\leq e^{b'-c'n}$.

   We note that for a cyclically reduced word $w$ of length $k$, there are $$\lp 2r-2\rp\lp 2r-1\rp^{\lp n-k\rp/2-1}$$ words of length $n$ whose cyclic reduction is $w$ if $n-k$ is even, and $0$ otherwise. This is because any word whose cyclic reduction is $w$ can be written as $a^{-1}wa$. If $i$ is the length of $a$, this word has length $k+2i$. The conditions on $a$ are that the first letter of $a$ is not the inverse of the last letter of $w$, or the first letter of $w$. These are different as $w$ is cyclically  reduced, giving $2r-2$ choices for the first letter of $a$. There are $2r-1$ choices for each of the remaining letters of $a$, giving $\lp 2r-2\rp\lp 2r-1\rp^{i-1}$ choices for $a$. Solving $n=k+2i$ for $i$ when $n-k$ is even gives one of the above counts, and the other comes from the fact that if $n-k$ is odd, there are no words of length $n$ whose cyclic reduction are of length $k$.  In either case, there are less than $\lp 2r-1\rp^{\lp n-k\rp/2}$ words of length $n$ whose cyclic reduction is $w$ of length $k$. 

   Therefore, we have that, \begin{align*}
    \frac{\# A_n^\circ}{\#\Free_n^\circ}&\leq \sum_{i=0}^n \frac{\lp 2r-1\rp^{i/2}\#B^\circ_{n-i}}{\lp 2r-1\rp^n}\\
    &\leq \sum_{i=0}^n \frac{\# B_{n-i}^\circ}{\lp 2r-1\rp^{n-i/2}}\\
    &=\sum_{i=0}^n \frac{2\# B_{n-i}^\circ}{2\lp 2r-1\rp^{n-i}\lp 2r-1\rp^{i/2} }\\
    \intertext{ Using the fact that $B$ is exponentially rare in $\mathcal C$, let $b,c,$ and $M$ be as in the conclusion of Lemma \ref{CM24converse}. Furthermore, assume that $c<\ln\lp\sqrt{2r-1}\rp$ since we can make $c$ smaller and the estimate $\frac{\#B_n^\circ}{\#\mathcal C_n^\circ}\leq e^{b-cn}$ still holds. Take $k_0>M$, and let $C=\max\{\#B_i^\circ\}_{i=0}^{k_0}$. As we know that $\#\mathcal C_n\leq 2\lp 2r-1\rp^n$, we have that for $n>k_0$,} 
    \frac{\#A_n^\circ}{\#\Free_n^\circ}&\leq 2 \sum_{i=0}^{n-k_0}\frac{\#B^\circ_{n-i}}{\#\mathcal C_{n-i}^\circ\lp 2r-1\rp^{i/2}}+\sum_{i=n-k_0}^n\frac{C}{\lp 2r-1\rp^{n-i/2}}\\
    &\leq 2\sum_{i=0}^{n-k_0}\frac{\#B_{n-i}^\circ}{\#C_{n-i}^\circ\lp 2r-1\rp^{i/2}}+\sum_{i=n-k_0}^n \frac{C}{\lp 2r-1\rp^{n-i/2}}\\
    &\leq 2\sum_{i=0}^{n-k_0}e^{b-c\lp n-i\rp-i\ln\lp \sqrt{2r-1}\rp}+\frac{k_0C}{\lp 2r-1\rp^{n/2-k_0}}\\
    &=2\sum_{i=0}^{n-k_0} e^{b-cn-\lp\ln\lp \sqrt{2r-1}\rp-c\rp i}+\frac{k_0 C}{\lp 2r-1\rp^{n/2-k_0}}\\
    &\leq 2\lp n-k_0\rp e^{b-cn}+k_0 C e^{\ln\lp k_0\rp}e^{-n\ln\lp \sqrt{2r-1}\rp}
    \end{align*}
    We leave it to the reader to verify that for any $c'<c$ there is a $b'$ so that for sufficiently large $n$,
    $$\frac{\#A_n^\circ}{\#\Free_n^\circ}\leq e^{b'-c'n}.$$
\end{proof}

Lemma \ref{DenseEnough} follows quickly from \cite{KSSGenericProperties}. We will use Lemmas \ref{DenseEnough} and \ref{hasKedges} to show that for a random peripheral structure on $\Free$, the associated Whitehead graphs have the properties necessary to apply Lemma \ref{Kconnectedimpliesnosmall}. 

\begin{lemma}[cf \cite{KSSGenericProperties}*{Proposition 6.2}]\label{DenseEnough}
    For any vertices $p,q$ in $\mathcal{WH}_{\mathcal{B}}\lp \{p\}\rp\{\langle w\rangle\}$, let $E_{p,q}$ denote the set of edges between $p$ and $q$. Let $\epsilon>0$, and let $Q\lp n,\epsilon\rp$ be the set of cyclically reduced words $w$ of length $n$ such that for every pair of vertices $\{p,q\}$ in the Whitehead graph $\mathcal{WH}_{\mathcal{B}}\lp \{p\}\rp\{\langle w\rangle\}$ $$\frac{\#E_{p,q}}{n}\in\lp \frac{1}{r\lp 2r-1\rp}-\epsilon,\frac{1}{r\lp 2r-1\rp}+\epsilon\rp.$$ Then the union $\bigcup\limits_{n\geq 0} Q\lp n,\epsilon\rp$ is exponentially generic in the set of cyclically reduced words.
\end{lemma}
\begin{proof}
    This is immediate from \cite{KSSGenericProperties}*{Proposition 6.2 (1)} with Lemma \ref{ShellOnly}.
\end{proof}

We say a word $w\in\bigcup\limits_{n\geq 0} C\lp n,\epsilon\rp$ is $\epsilon$\emph{-pseudorandom}, after the terminology in \cite{CalegariWalker}*{Definition 3.2.2}.

\begin{lemma}\label{hasKedges}
    Let $k>0$. Let $B_k$ be the set of reduced words in $\Free$ such that the cyclic reduction of every $b\in B_k$ contains $k$ copies of each reduced word of length $2$. Then $B_k$ is exponentially generic in $\Free$.
\end{lemma}
\begin{proof}
    Let $S$ be the set of elements of $\Free$ whose middle third survives cyclic reduction. Let $P$ be the set of words whose cyclic reduction is $\frac{1}{2r\lp 2r-1\rp}-$ pseudorandom. Let $L_k$ be the set of words of length at least  $15r\lp2r-1\rp k$. Consider $x\in S\cap P\cap L$. Then the cyclic reduction $x'$ has length at least $5r\lp 2r-1\rp k$, and is $\frac{1}{2r\lp 2r-1\rp}$-pseudorandom, so it contains at least \begin{align*}
        \lp \frac{1}{r\lp 2r-1\rp}-\frac{1}{2r\lp 2r-1\rp}\rp\lp 5r\lp 2r-1\rp k\rp &=\frac{1}{2}\frac{ 5r\lp 2r-1\rp k}{\lp 2r\rp\lp 2r-1\rp}\\[1em]
        &=\frac{5r\lp 2r-1\rp k}{4r\lp 2r-1\rp}\\[1em]
        &>k
    \end{align*}
    copies of every reduced word of length $2$. Therefore, $S\cap P\cap L_k\subseteq B_k$. Therefore, it is sufficient to show that $S\cap P\cap L_k$ is exponentially generic.

    The fact that $S$ is exponentially generic is \cite{CashenManning}*{Lemma 2.7}. The fact that $P$ is exponentially generic follows from Lemma \ref{rarereduction} and Lemma \ref{DenseEnough}. The fact $L_k$ is exponentially generic follows from exponential growth of the free group, but the estimate will be shown here for completeness. Let $N_k=30r\lp 2r-1\rp k$. Take $l>N_k$. Take $\alpha=1+\sum_{j=1}^{N_k-1} 2r\lp 2r-1\rp^{j}$, the number of words of length less than $N_k$. The number of words of length between $N_k$ and $l$ inclusive is $\sum_{j=N_k}^{l}2r\lp 2r-1\rp^j$. Therefore, \begin{align*}
        1-\frac{\#\{w\in L_k|\|w\|_{\mathcal B}\leq l\}}{\#\{w\in \Free|\|w\|_{\mathcal B}\leq l\}}&=1-\frac{\sum_{j=N_k}^{l} 2r\lp 2r-1\rp^j}{\alpha+\sum_{j=N_k}^{l} 2r\lp 2r-1\rp^j}\\[1em]
        &=\frac{\alpha+\sum_{j=N_k}^{l} 2r\lp 2r-1\rp^j-\sum_{j=N_k}^{l} 2r\lp 2r-1\rp^j}{\alpha+\sum_{j=N_k}^{l} 2r\lp 2r-1\rp^j}\\[1em]
        &=\frac{\alpha}{\alpha+\sum_{j=N_k}^{l} 2r\lp 2r-1\rp^j}\\[1em]
        &=\frac{1}{1+\frac{\sum_{j=N_k}^{l} 2r\lp 2r-1\rp^j}{\alpha}}\\[1em]
        &<\frac{\alpha}{\sum_{j=N_k}^{l} 2r\lp 2r-1\rp^j}.
    \end{align*}

    Since $\sum_{j=N_k}^{l} 2r\lp 2r-1\rp^j$ grows exponentially in $l$, and $\alpha$ remains constant, we have that $L_k$ is exponentially generic in $\Free$.

    As the intersection of exponentially generic sets is exponentially generic, and a set containing an exponentially generic set is exponentially generic, $B_k$ is exponentially generic in $\Free$.
\end{proof}

\begin{lemma}\label{fourfull}
    Let $H_k$ be the set of elements of $\Free$ such that for every $h\in H_k$, the cyclic reduction $h'$ is not a proper power, is $4$-full, and contains at least $k$ copies of every reduced length $2$ word as a subword. Then $H_k$ is exponentially generic in $\Free$.
\end{lemma}
\begin{proof}
      Let $B_k$ be as in Lemma \ref{hasKedges}. Let $w$ be a 4-full word, and let $W$ be the set of words in $\Free$ which contain $w$ in their middle third. Let $N$ be the set of words whose cyclic reduction are not proper powers. We note that if $x\in \Free$ contains a $4$-full word, then $x$ is $4$-full. Then, $B_k\cap W\cap N \subseteq H_k$. We know that $B_k$ is exponentially generic by Lemma \ref{hasKedges}, $W$ is exponentially generic by \cite{CashenManning}*{Proposition 2.8}, and $N$ is exponentially generic by \cite{KSSGenericProperties}*{Theorem B(1)}. Therefore $H_k$ is exponentially generic in $\Free$ as claimed.
\end{proof}

We now prove Theorem \ref{NoSizeK}. 
\begin{proof}
    Take $\Per'$ to be a random collection of words of length at most $n$. Let $\Per'$ be the induced peripheral structure. Let $H_k$ be as in Lemma \ref{fourfull}. As $H_k$ is exponentially generic in $\Free$, $\Per$ contains a word $w$ whose which is 4-full, not a proper power, and contains at least $k$ copies of every reduced length 2 subword at least $k$ times with probability $1-e^{a-bn}$ for some $a\in \mathbb R$ and $b>0$. By Lemma \ref{Kconnectedimpliesnosmall}, we then have that any cut set of $\partial\FreeRelPer$ contains at least $k$ elements.
\end{proof}

\bibliography{bibliography}

\end{document}